\documentclass[12pt]{amsart}
\usepackage{amsmath, amssymb, amsthm, color, bbold, url, extarrows}
\usepackage{MnSymbol}
\usepackage[margin=1in]{geometry}

\usepackage{dsfont}

\usepackage{graphicx}
\usepackage[all, cmtip]{xy}
\usepackage[utf8]{inputenc}
\usepackage{url}

\usepackage{lineno}

\usepackage{enumerate}

\usepackage{hyperref}
\hypersetup{
    bookmarks=true,         
    unicode=false,          
    pdftoolbar=true,        
    pdfmenubar=true,        
    pdffitwindow=false,     
    pdfstartview={FitH},    
    pdftitle={My title},    
    pdfauthor={Author},     
    pdfsubject={Subject},   
    pdfcreator={Creator},   
    pdfproducer={Producer}, 
    pdfkeywords={keyword1} {key2} {key3}, 
    pdfnewwindow=true,      
    colorlinks=true,       
    linkcolor=black,          
    citecolor=cyan,        
    filecolor=magenta,      
    urlcolor=cyan           
}

\newtheorem{thm}{Theorem}[section]

\newtheorem{cor}[thm]{Corollary}
\newtheorem{lemma}[thm]{Lemma}
\newtheorem{prop}[thm]{Proposition}

\theoremstyle{rmk}
\theoremstyle{definition}

\newtheorem{rmk}[thm]{Remark}

\numberwithin{equation}{section}

\def\C{\mathbb C}

\def\kk{\mathbb k}

\def\ga{\mathfrak g}

\def\ra{\mathfrak r}
\def\rau{{\mathfrak r}_{\rm u}}

\def\aa{\mathfrak a}

\def\ha{\mathfrak h}
\def\ea{\mathfrak e}
\def\fa{\mathfrak f}

\def\gl{\mathfrak{gl}}

\def\wt{\widetilde}
\def\ol{\overline}

\def\GL{\operatorname{GL}}

\def\SL{\operatorname{SL}}

\def\Sp{\operatorname{Sp}}
\def\SO{\operatorname{SO}}
\def\OO{\operatorname{O}}
\def\Ru{\operatorname{R_u}}

\def\P{\operatorname{P}}

\def\Lie{\operatorname{Lie}}

\DeclareMathOperator{\Ker}{\operatorname{Ker}}

\renewcommand{\mathbb}{\mathds}

\newcommand{\CC}{\mathbb{C}}

\newcommand{\ZZ}{\mathbb{Z}} 

\newcommand{\group}{\mathrm}

\newcommand{\E}{\group{E}}
\newcommand{\F}{\group{F}}
\newcommand{\Zen}{\mathrm{Z}}

\newcommand{\Mat}{\mathrm{Mat}}

\renewcommand{\rho}{\varrho}
\renewcommand{\phi}{\varphi}

\newcommand{\pre}{prehomogeneous}

\newcommand{\et}{{\'e}tale}

\definecolor{todo}{rgb}{1,0,0}

\author[Burde]{Dietrich Burde}
\address{Dietrich Burde,
Faculty of Mathematics\\
University of Vienna\\
Oskar-Morgenstern-Platz 1\\
1090 Vienna\\
Austria}
\email{dietrich.burde@univie.ac.at}

\author[Globke]{Wolfgang Globke}
\address{Wolfgang Globke,
School of Mathematical Sciences\\
The University of Adelaide\\
SA 5005\\
Australia}
\email{wolfgang.globke@adelaide.edu.au}

\author[Minchenko]{Andrei Minchenko}
\address{Andrei Minchenko,
Faculty of Mathematics\\
University of Vienna\\
Oskar-Morgenstern-Platz 1\\
1090 Vienna\\
Austria}
\email{andrei.minchenko@univie.ac.at}

\subjclass[2010]{Primary 32M10; Secondary 20G05, 20G20}

\title[{\'E}tale representations]{{\'E}tale representations for reductive algebraic groups with factors $\boldsymbol{\Sp_n}$ or $\boldsymbol{\SO_n}$}

\begin{document}
	\maketitle
%
%

\begin{abstract}
An \'etale module for a linear algebraic group $G$ is a
complex vector space $V$ with a rational $G$-action on $V$
that has a Zariski-open orbit and $\dim G=\dim V$.
Such a module is called super-\'etale if the stabilizer of a
point in the open orbit is trivial.
Popov (2013) proved that
reductive algebraic groups admitting super-\'etale modules are
special algebraic groups.
He further conjectured that a reductive group admitting a
super-\'etale module
is always isomorphic to a product of general linear groups.
Our main result is the construction of counterexamples
to this conjecture, namely a family of
super-\'etale modules for groups with a factor $\Sp_n$
for arbitrary $n\geq1$.
A similar construction provides a family of \'etale modules
for groups with a factor $\SO_n$, which shows that groups with
\'etale modules with non-trivial stabilizer are not necessarily
special.
Both families of examples are somewhat surprising in light of the
previously known examples of \'etale and super-\'etale modules
for reductive groups.
Finally, we show that the exceptional groups $\F_4$ and $\E_8$
cannot appear as simple factors in the maximal semisimple subgroup
of an arbitrary Lie group with a linear \'etale representation.
\end{abstract}

\section*{Introduction}

An \emph{\'etale module} $(G,\rho,V)$ for an algebraic
group $G$ is a finite-dimensional complex vector space $V$
together with a rational representation $\rho:G\to\GL(V)$
such that $\rho(G)$ has a Zariski-open orbit in $V$ and
$\dim G=\dim V$.
In particular, the stabilizer $H$ of any point in the open orbit
is a finite subgroup of $G$.
If $H$ is the trivial group, the module is called
\emph{super-\'etale}.
Similarly we call the representation $\rho$ \'etale or
super-\'etale, respectively.
More generally, one can study affine \'etale representations
(that is, representations by affine transformations),
but for rational representations of reductive algebraic
groups these are equivalent to linear ones via affine changes
of coordinates.
As we are primarily interested in this case, we shall restrict ourselves to linear representations.

The existence of an affine \'etale representation for a given
group $G$
implies the existence of a left-invariant flat affine
connection on $G$, and these structures appear in many different
contexts in mathematics. For the specifics of this relationship
and a survey of applications,
see Burde \cite{Burde96, Burde06}, Baues \cite{Baues99} and the
references therein.
The primary motivation for the present work is Popov's study of linearizable subgroups of
the Cremona group on affine $n$-space (those that are conjugate to 
linear group within the Cremona group).
Subgroups for which a super-\'etale module exists,
called \emph{flattenable} groups by Popov, allow particularly
convenient criteria to decide their linearizability, compare the
results in \cite[Section 2]{Popov13}.
Incidentally, a flattenable group $G$ is precisely a group that
admits a rational super-\'etale module.
Popov \cite[Lemma 2]{Popov13} proved (in our terminology):

{\em A reductive algebraic group admitting a super-\'etale module
is a special algebraic group.}

By definition, $G$ is \emph{special} (in the sense of Serre)
if every principal $G$-bundle is locally trivial in the \'etale
topology. Serre \cite[4.1]{Serre58} showed that every special group
is connected and linear, and that reductive groups with maximal
connected semisimple subgroup
\[
S(G) = \SL_{n_1}\times\cdots\times\SL_{n_k}\times\Sp_{m_1}\times\cdots\times\Sp_{m_j}
\]
are special.
A result of
Grothendieck \cite[Th\'eor\`eme 3]{Groth58} then implies that an
affine algebraic group $G$ is special if and only if a maximal
connected semisimple subgroup is isomorphic to a group of this
type.
%
%
This result and the available examples lead Popov to make the
following conjecture:

{\em A reductive algebraic group $G$ has a rational super-\'etale
module if and only if}
\[
G\cong\GL_{n_1}\times\cdots\times\GL_{n_k}.
\]
Clearly, every group $\GL_{n_1}\times\cdots\times\GL_{n_k}$
has a super-\'etale module $\Mat_{n_1}\oplus\ldots\oplus\Mat_{n_k}$
on which it acts factorwise by matrix multiplication.
In previously available classification
results on \'etale modules for reductive algebraic groups $G$,
the only simple groups appearing as factors in $G$ are $\SL_n$
and $\Sp_2$ (see Burde and Globke \cite[Section 5]{BG16}
for a summary).
This suggests the more general questions of whether
{\em in a reductive algebraic group with a rational
super-\'etale module, all simple factors are either $\Sp_2$ or
$\SL_n$ for certain $n\geq 2$.}
Somewhat surprisingly, this (and thus Popov's original
conjecture) turns out to be false. Our main result is the
existence of counterexamples to this conjecture, constructed in Section \ref{sec:Spn} below show. These examples consist of a
family of super-\'etale modules for reductive groups
$G=\Sp_n\times\GL_{2n-1}\times\cdots\times\GL_1$ for any $n\geq 1$. 
So in fact any factor $\SL_n$ or $\Sp_n$ for any $n\geq 1$ can
appear in a group with a super-\'etale module.
One might now be tempted to ask whether every special reductive
algebraic group admits a super-\'etale module, but this can
immediately be ruled out by comparison with classification results of reductive groups
with few simple factors, see again \cite[Section 5]{BG16}.

Knowing that algebraic groups with super-\'etale modules are
special, one can further suspect that the same holds for groups
with \'etale modules that have non-trivial stabilizer.
%
%
Again we find the surprising answer that this is not true.
In Section \ref{sec:SOn} below we construct a family of \'etale
modules for reductive groups
$G=\SO_n\times\GL_{n-1}\times\cdots\times\GL_1$ for any
$n\geq 2$.
These are the first known examples of \'etale
modules for groups with a simple factor $\SO_n$ for any number
$n\geq2$.

These two families are the first known examples of \'etale modules
for reductive groups containing factors $\Sp_n$ or $\SO_n$ for
arbitrary $n>2$.
This still leaves the question of whether there exist \'etale
modules for reductive groups with exceptional simple groups as
factors.
In Section \ref{sec:no_F4E8}, we show in a much more general
setting that a simple Lie group whose complexified Lie algebra is
one of the exceptional algebras $\fa_4$ or $\ea_8$ cannot appear
among the simple factors in a maximal semisimple subgroup of a
Lie group with a linear \'etale representation, not necessarily algebraic
(here, \emph{\'etale} means that the action has an orbit that is
open in the standard topology of the module).
For the other exceptional groups, this question remains open.

A remark on the previously available classification results on
\'etale modules is in order.
As these results use the classification results on
prehomogeneous modules due to Sato, Kimura and others (see Kimura's book \cite[Chapter 7]{Kimura03} and references
therein),
they very often rely on Lie algebraic methods.
In most cases it is not immediately clear from their
classifications
whether the generic stabilizers are trivial, although many
generic stabilizers (not just their identity component) are
explicitely given in the appendix of \cite{Kimura03}.

\subsection*{Notations and conventions}

All algebraic groups, such as $\GL_n$, $\SL_n$, $\SO_n$
and $\Sp_n$, are considered over the complex numbers unless
otherwise stated.
We follow the convention that $\Sp_n$ means the symplectic
group on $\CC^{2n}$.
The notation $\Lie G$ means the Lie algebra of 
a group $G$, we will also use the corresponding gothic letter
$\ga$.
The identity component of an algebraic group $G$ is denoted by 
$G^\circ$.
$\Mat_{m,n}$ denotes the space of complex $m\times n$-matrices,
and if $m=n$ we simply write $\Mat_n$.
The identity matrix in $\Mat_n$ is denoted by $I_n$.
The transpose of a matrix $A$ is denoted by $A^\top$.
The canonical basis vectors of $\CC^n$ are denoted by
$e_1,\ldots,e_n$.

For any algebraic group $G$, let $\Zen(G)$, $L(G)$, and $\Ru(G)$
denote the center, a maximal connected reductive subgroup, and the unipotent radical of $G$, respectively.
Then $G$ is the semidirect product $G=L(G)\cdot\Ru(G)$.
Write $S(G)$ for a maximal connected semisimple subgroup
of $G$, the commutator subgroup of $L(G)$.
Note that $L(G)$ and $S(G)$ are unique up to conjugation.

\subsection*{Acknowledgements}
The authors would like to thank Vladimir Popov
and Alexander Elashvili
for helpful discussions and comments, and also the anonymous
referees for many helpful remarks and suggestions to improve
the article.
Dietrich Burde acknowledges support by the Austrian Science Foundation FWF, grant P28079 and grant I3248.
Andrei Minchenko acknowledges support by the Austrian Science Foundation FWF, grant P28079.
Wolfgang Globke acknowledges support by the Australian Research Council,
grant DE150101647.

\section{Preliminaries on prehomogeneous modules}\label{sec:preliminaries}

A module $(G,\rho,V)$, or $(G,V)$ for short, for an algebraic group
$G$ with a rational representation $\rho:G\to\GL(V)$ on a
finite-dimensional complex vector space $V$ is called a
\emph{prehomogeneous module} if $\rho(G)$ has a Zariski-open
orbit in $V$. In this case, $\dim G\geq\dim V$. More precisely,
if $x\in V$ is a \emph{point in general position}, that is, it
lies in the open orbit of $G$, and $G_x$ its stabilizer subgroup,
then
\[
\dim V = \dim G - \dim G_x.
\]
The stabilizer $H=G_x$ of any point $x$ in the open orbit is
called the \emph{generic stabilizer} of $(G,\rho,V)$.
A prehomogeneous module is \emph{\'etale} if $H^\circ=\{1\}$
(equivalently, if $\dim G=\dim V$).
An \'etale module $(G,V)$ is called \emph{super-\'etale} if
$H=\{1\}$.

%
%

See Burde and Globke \cite[Proposition 4.1]{BG16} for a proof of
the
following result which we will use frequently without further
reference:

\begin{prop}\label{prop:stabilizers}
The following conditions are equivalent:
\begin{enumerate}
\item[\rm (1)]
$(G,\rho_1\oplus\rho_2,V_1\oplus V_2)$ is an \'etale
module.
\item[\rm (2)]
$(G,\rho_1,V_1)$ is prehomogeneous and
$(H^\circ,\rho_2|_H,V_2)$
is an \'etale module, where $H^\circ$ denotes the connected
component of the generic stabilizer of $(G,\rho_1,V_1)$.
\end{enumerate}
Equivalence also holds if each ``\'etale'' is replaced by
``prehomogeneous''.
\end{prop}

Two modules $(G_1,\rho_1,V_1)$ and
$(G_2,\rho_2,V_2)$ are called \emph{equivalent} if there exists
an isomorphism of algebraic groups $\psi:\rho_1(G_1)\to\rho_2(G_2)$
and a linear isomorphism $\phi:V_1\to V_2$ such that
$\psi(\rho_1(g))\phi(x)=\phi(\rho_1(g)x)$ for all $x\in V_1$
and $g\in G_1$.


Let $m>n\geq 1$ and $\rho:G\to\GL(V)$ be an $m$-dimensional
rational representation of an algebraic group $G$, and let
$\rho^*$ be the dual representation for $\rho$.
Then we say that the modules
\[
\bigl( G\times \GL_n,\ \rho\otimes\omega_1,\ V \otimes \CC^n \bigr)
\quad \mathrm{and}\quad
\bigl( G\times \GL_{m-n},\ \rho^*\otimes\omega_1,\ V^{*} \otimes \CC^{m-n} \bigr)
\]
are \emph{castling transforms}\index{castling transform} of each
other.
More generally, we say two modules $(G_1,\rho_1,V_1)$ and
$(G_2,\rho_2,V_2)$ are \emph{castling-equivalent}
if $(G_1,\rho_1,V_1)$ is equivalent to a module
obtained after a finite number of castling transforms from
$(G_2,\rho_2,V_2)$.
A module $(G,\rho,V)$ is called \emph{reduced} (or \emph{castling-reduced}) if $\dim V\leq \dim V'$ for every castling transform
$(G,\rho',V')$ of $(G,\rho,V)$.
Sato and Kimura \cite[\S 2]{SK77} proved that prehomogeneity and generic
stabilizers are preserved by castling transforms.

\section{\'Etale modules for groups with factor $\Sp_n$ or $\SO_n$}\label{sec:SpSO}

In this section we will construct two families of \'etale modules
for reductive algebraic groups $G$.
In the first family, $G$ contains a simple factor $\Sp_n$,
$n\geq 1$, and theses modules are even super-\'etale, thus proving
that groups with super-\'etale modules are not restricted to
products of special linear groups.
In the second family, $G$ contains a factor $\SO_n$, $n\geq 2$.
This proves that groups with \'etale modules (but possibly
non-trivial stabilizer) do not have to be special in the sense
of Serre.
Moreover, these are the first known examples of \'etale modules
for reductive algebraic groups that contain factors $\Sp_n$ or
$\SO_n$ for arbitrary $n>2$.

We need some preparations.
Suppose $G$ is an algebraic group of the form
\[
G=G_m\times G_{m-1}\times\cdots\times G_1,
\]
where $G_k\subseteq\GL_k$.
The vector space
\begin{equation}
E_m = \Mat_{m,m-1}\oplus\Mat_{m-1,m-2}\oplus\ldots\oplus\Mat_{2,1}
\label{eq:Em}
\end{equation}
becomes a $G$-module for the action defined as follows:
An element $A=(A_m,\ldots,A_1)\in G$ acts on
$X=(X_{m-1},\ldots,X_1)\in E_m$ by
\begin{equation}
A.X
=
(A_m X_{m-1} A_{m-1}^\top,\
A_{m-1} X_{m-2} A_{m-2}^\top,\
\ldots,\
A_2 X_1 A_1^\top).
\label{eq:action}
\end{equation}
Note that
\begin{equation}
\dim E_m
=
\sum_{k=1}^{m-1} (k+1)k
=
\frac{m(m-1)}{2}+\sum_{k=1}^{m-1} k^2.
\label{eq:dimE}
\end{equation}

\subsection[Super-\'etale modules for groups with factor $\Sp_n$]{Super-\'etale modules for groups with factor $\boldsymbol{\Sp_n}$}\label{sec:Spn}

We wish to construct a family of super-\'etale
modules for the group
\[
G = \Sp_n \times \GL_{2n-1}\times\cdots\times\GL_1.
\]
We define a symplectic form $\omega$
in terms of the canonical basis of $\CC^{2n}$ by
\begin{align*}
\omega(e_{2j-1},e_{2j})&=1 \quad\text{ for }j=1,\ldots,n, \\
\omega(e_{2j-1},e_k)&=0=\omega(e_{2j},e_k) \quad\text{ for }k\neq 2j,2j-1.
\end{align*}
Define subspaces $F_k=\mathrm{span}\{e_1,\ldots,e_k\}$
of $\CC^{2n}$ for $k=1,\ldots,2n$.

Let $\Sp_n\subset\GL_{2n}$ denote the symplectic group that
preserves the symplectic form $\omega$.
Then for every $A\in\Sp_n$ and $k=1,\ldots,2n$, we have
$A e_{k}^\perp \perp A e_k$.

We can identify $F_{k+1}\otimes F_k$ with
$\CC^{k+1}\otimes\CC^k\cong\Mat_{k+1,k}$.
With $E_{2n}$ from \eqref{eq:Em}, introduce the $G$-module
\[
V = \CC^{2n}\oplus E_{2n}.
\]
where $G$ acts on $\CC^{2n}$ by the standard action of $\Sp_n$
and $G$ acts on $E_{2n}$ by \eqref{eq:action},
for $G_{2n}=\Sp_n$, $G_{2n-1}=\GL_{2n-1}$,\ldots, $G_1=\GL_1$.

We have
\begin{equation}
\begin{split}
\dim G
&=2n^2+n + \sum_{k=1}^{2n-1}k^2
=2n+\frac{2n(2n-1)}{2} + \sum_{k=1}^{2n-1}k^2\\
&=2n+\sum_{k=1}^{2n-1}k + \sum_{k=1}^{2n-1}k^2
=2n+\dim E_{2n}
=\dim V.
\end{split}
\label{eq:dimGdimV_Spn}
\end{equation}

We will prove by induction on $n$ that $V$ is super-\'etale for
$G$. We only need to show that the generic stabilizer of the
$G$-action is trivial, then it follows from \eqref{eq:dimGdimV_Spn}
that $G$ has an open orbit.

In the case $n=1$, $G\cong\SL_2\times\GL_1$
and $V=\Mat_2$, where $\SL_2$ acts by
matrix multiplication and $\GL_1$ by scalar multiplication of the
second column of a $2\times 2$-matrix.
One verifies directly that this is a super-\'etale module, and so
this confirms the initial case for the induction:

\begin{lemma}\label{lem:induction_start_Spn}
For $n=1$, the given action of $G=\Sp_1\times\GL_1$ on
$V=\CC^2\oplus\CC^2$ is \'etale and has trivial stabilizer
at the point $(e_1,e_2)\in V$.
\end{lemma}

For the induction step,
consider the action of $\Sp_n\times\GL_{2n-1}$
on $\CC^{2n}\oplus(F_{2n}\otimes F_{2n-1})$ first.
We can identify this space with $\Mat_{2n}$,
the action of $(A,B)\in\Sp_n\times\GL_{2n-1}$ given by
\[
(A,B).X =
A X \left(\begin{smallmatrix}
B^\top & 0 \\
0&  1
\end{smallmatrix}\right),
\quad
X\in\Mat_{2n,2n}.
\]
As a point in general position, choose the identity matrix
$X_0=I_n$. Then, if
\[
A I_n \left(\begin{smallmatrix}
B^\top & 0 \\
0&  1
\end{smallmatrix}\right)
=I_n,
\]
it follows that
\[
A = \begin{pmatrix}
A_1 & 0 \\
0 &  1
\end{pmatrix}\in\Sp_n
\]
with $A_1=(B^\top)^{-1}\in\GL_{2n-1}$.
Recall that $A e_{2n}=e_{2n}$ implies $A e_{2n}^\perp=e_{2n}^\perp$,
and the form of the matrix $A$ thus requires $A F_{2n-2}=F_{2n-2}$.
Also, $A e_{2n-1}=e_{2n-1}$ since $A$ also preserves $F_{2n-2}^\perp$.
Hence
\[
A = \begin{pmatrix}
A_0 & 0 \\
0&  I_2
\end{pmatrix}\in\Sp_n,
\quad A_0\in\Sp_{n-1}.
\]
This proves:

\begin{lemma}\label{lem:SpGL}
The stabilizer $H$ of the $\Sp_n\times\GL_{2n-1}$-action
on $\CC^{2n}\oplus(\CC^{2n}\otimes\CC^{2n-1})$ at the point $X_0$
is given by
\[
H = \Bigl\{\Bigl(
\begin{pmatrix}
A_0 & 0 \\
0&  I_2
\end{pmatrix},
\begin{pmatrix}
A_0^{-1} & 0 \\
0&  1
\end{pmatrix}
\Bigr)
\ \Bigl|\
A_0\in\Sp_{n-1}
\Bigr\}
\cong\Sp_{n-1}.
\]
Hence the stabilizer $H_{2n-1}$ of the $G$-action on the submodule
$\CC^{2n}\oplus(\CC^{2n}\otimes\CC^{2n-1})$ of $V$ at the point
$X_0$ is
\[
H_{2n-1}
=
\Sp_{n-1}\times \GL_{2n-2}\times\cdots\times\GL_1,
\]
with the embedding of $\Sp_{n-1}$ in $G$ given as above.
\end{lemma}

Consider the first summand $W$ in $E_{2n-1}$,
\[
W=F_{2n-1}\otimes F_{2n-2}=\Mat_{2n-1,2n-2}
\]
where the $H_{2n-1}$-action is given by the action of the factor
$\Sp_{n-1}\times\GL_{2n-2}$.
Here, $\Sp_{n-1}$ is identified with
the projection of the stabilizer of $\Sp_n\times\GL_{2n-1}$ to
$\GL_{2n-1}$ (see Lemma \ref{lem:SpGL}), and
this projection acts on the subspace $F_{2n-2}\subset F_{2n-1}$
and trivially on its complement in $F_{2n-1}$.
Thus we can rewrite the module $W$ as
\begin{align*}
W&=(F_{2n-2}\oplus\CC e_{2n-1})\otimes F_{2n-2}=W_1\oplus W_2, \\
W_1 &=F_{2n-2}\otimes F_{2n-2}\cong\Mat_{2n-2},\\
W_2 &=\CC\otimes F_{2n-2}=F_{2n-2}\cong\CC^{2n-2},
\end{align*}
where $(A,B)\in\Sp_{n-1}\times\GL_{2n-2}$ acts on $X\in W_1$ by
$X\mapsto A X B^\top$ and on $y\in W_2$ by $y\mapsto By$.

Choose $X_1=I_{2n-2}$ as a point in general position for the
action on $W_1$.
The stabilizer of this action is again a diagonally embedded copy of
$\Sp_{n-1}$ in $\Sp_{n-1}\times\GL_{2n-2}$.
Identifying this copy once again with its projection to $\GL_{n-2}$,
we have an $\Sp_{n-1}$-action on $W_2\cong\CC^{n-2}$ by
left multiplication.

\begin{lemma}\label{lem:H2n-2}
The stabilizer of the $H_{2n-1}$-action at the point $X_1=I_{2n-2}$
in the module $W_1$ is the group
\[
H_{2n-2} = \Sp_{n-1}\times \GL_{2n-3}\times\cdots\times \GL_1,
\]
where the $\Sp_{n-1}$-action on $W_2=\CC^{2n-2}$ is by
left multiplication.
\end{lemma}

In order for $E_{2n-1}=W\oplus E_{2n-2}$ to be \'etale for the $H_{2n-1}$-action
(and thus the original module $V$ to be \'etale for the
$G$-action),
the stabilizer $H_{2n-2}$ must have an \'etale action on
\[
W_2 \oplus E_{2n-2}
=\CC^{2n-2} \oplus E_{2n-2}.
\]
Observe now that $\CC^{2n-2} \oplus E_{2n-2}$ is of the same form as the original
module $V$, and $H_{2n-2}$ is of the same form as the original
group $G$, with $n$ replaced by $n-2$.
Now we can apply the induction hypo\-thesis to conclude that
the $H_{2n-2}$-action on $V_{2n-2}$ and thus the $G$-action on
$V$ is super-\'etale (where we assume that all points in general
position are chosen similarly to $X_0$, $X_1$ above).

\begin{thm}\label{thm:superetale_Spn}
The module $(\Sp_n \times \GL_{2n-1}\times\cdots\times\GL_1,\CC^{2n}\oplus E_{2n})$ with the action given above
is a super-\'etale module.
\end{thm}

\begin{rmk}
For $n=2$, $(G,\rho,V)$ in Theorem \ref{thm:superetale_Spn}
can be viewed as a variation of an example given by Helmstetter
\cite[p.~1090]{Helm71}, which is the module
\[
G = \Sp_2\times\GL_3\times\GL_2\times\GL_1\times\GL_1,
\quad
V = \CC^4 \oplus (\CC^4\otimes\CC^3) \oplus (\CC^3\otimes\CC^2) \oplus \CC^3
\]
where the last copy of $\CC^3$ is identified with the space of
traceless $2\times 2$-matrices, and the action of $G$ is given by
\[
(A,B,C,\alpha,\beta).(x,Y,Z,U)
=
(\alpha Ax, AYB^\top, BZC^\top, \beta CUC^{-1}).
\]
This module is \'etale, but it is not super-\'etale, since the
action of $\GL_2$ on the last copy of $\CC^3$ has a non-connected
generic stabilizer.
\end{rmk}

\begin{rmk}
A second family of super-\'etale modules appears in the
construction of this section, namely the group
$\Sp_{n}\times\GL_{2n}\times\cdots\times\GL_1$ acting on the
module $E_{2n}$.
This group appears as the stabilizer in Lemma
\ref{lem:SpGL} (for $n-1$), where the module is the
module complement
of $\CC^{2n}\oplus(\CC^{2n}\otimes\CC^{2n-1})$ in this lemma.
\end{rmk}

\subsection[\'Etale modules for groups with factor $\SO_n$]{\'Etale modules for groups with factor $\boldsymbol{\SO_n}$}\label{sec:SOn}

We wish to construct a family of \'etale modules
for the group
\[
G = \SO_n \times \GL_{n-1}\times\cdots\times\GL_1,
\]
where we take $\SO_n$ to be the subgroup of $\SL_n$
preserving the bilinear
form represented by the identity matrix $I_n$.

%
Let $n\geq 2$. Consider the $G$-module $V=E_n$ with the action given
by \eqref{eq:action}, where $G_n=\SO_n$, $G_{n-1}=\GL_{n-1}$,\ldots, $G_1=\GL_1$.
We have
\begin{equation}
\dim G = \frac{1}{2}n(n-1) + \sum_{k=1}^{n-1} k^2
=\sum_{k=1}^{n-1} k + \sum_{k=1}^{n-1} k^2
=\dim E_n.
\label{eq:dimGdimV_SOn}
\end{equation}
In order to verify that $V$ is an \'etale module for $G$,
we only need to show that the connected component $H^\circ$
of the generic stabilizer $H$ is trivial. Then it follows
from \eqref{eq:dimGdimV_SOn} that $G$ has an open orbit and the
action is \'etale.


\begin{lemma}\label{lem:sogl}
The stabilizer $H_1$ of
$\SO_n\times\GL_{n-1}$ on the module $\Mat_{n,n-1}$ at the point
in general position $X_1=\left(\begin{smallmatrix}
I_{n-1}\\0\ldots0\end{smallmatrix}\right)$
is spanned by the elements
$(A,A_0)\in\SO_n\times\GL_{n-1}$ with
\[
A_0\in\OO_{n-1}
\quad
\text{ and }
\quad
A = \begin{pmatrix}
A_0 & {0} \\
{0} & \alpha
\end{pmatrix}\in\SO_n
\]
where $\alpha=\det(A_0)^{-1}$.
In particular,
\[
H_1\cong\OO_{n-1}.
\]
\end{lemma}
\begin{proof}
Let $(A,B)\in\SO_n\times\GL_{n-1}$,
and let $A_0$ be the upper left $(n-1)\times(n-1)$-block of $A$
and $a_n$ the first $n-1$ entries in the last row of $A$.
Then $AX_1 B^\top=X_1$ is equivalent to
$A_0^{-1} = B^\top$, $a_n={0}$,
and as $A_0$ is orthogonal, this gives the required form of the
stabilizer $H_1$ of $X_1$.
\end{proof}

The identity component $H_1^\circ\cong\SO_{n-1}$ of the generic
stabilizer of $(G,V)$
acts on the next summand $\Mat_{n-1,n-2}$ in $E_n$
via its injective projection to the
$\GL_{n-1}$-factor.
But this is identical to the left multiplication of $\SO_{n-1}$
on $\Mat_{n-1,n-2}$.
So we are now looking at the action of
\[
\SO_{n-1}\times\GL_{n-1}\times\cdots\times\GL_1
\]
given by \eqref{eq:action} on $E_{n-1}$.
When choosing a point in general position for this action as in
Lemma \ref{lem:sogl}, we can apply induction on $n$ to conclude that
this module is \'etale. Moreover, Lemma \ref{lem:sogl} for $n=2$
takes care of the initial case,
that is,
the action of the abelian group $\SO_2\times\GL_1$ on
$V=\CC^2$ given by $(A,\lambda)\mapsto \lambda A x$, $x\in\CC^2$,
is \'etale with generic stabilizer $H\cong\ZZ_2$.

So we have shown:

\begin{thm}\label{thm:SOetale}
Let $n\geq 2$.
The module $(\SO_n \times \GL_{n-1}\times\cdots\times\GL_1, E_n)$
with the action given by \eqref{eq:action}
is an \'etale module.
\end{thm}

\section{\'Etale Lie algebras over fields of characteristic $0$}
\label{sec:base_change}

Let $\kk$ be a field of characteristic $0$. Recall that a linear Lie algebra $\ga\subset\gl_n(\kk)$ is called \emph{algebraic} if there is a $\kk$-defined linear algebraic group $G\subset\GL_n$ such that $\ga=(\Lie G)(\kk)$.
The Lie algebra $\ga$ is called \emph{prehomogeneous} if there is a
point $o\in\kk^n$ such that the map $\beta:\ga\to\kk^n$, $X\mapsto Xo$ is a surjective homomorphism of vector spaces, and $\ga$ is
called \emph{\'etale} if $\beta$ is an isomorphism.
 
\begin{prop}\label{prop:hull}
Let $\ga\subset\gl_n(\kk)$ be a \pre\ Lie algebra with generic stabilizer $\ha$.
Then there is a \pre\ algebraic Lie algebra $\wt\ga\subset\gl_n(\kk)$ with generic stabilizer $\wt\ha$ such that
$[\ga,\ga]=[\wt\ga,\wt\ga]$ and $\wt\ha=\ha\cap[\ga,\ga]$.
\end{prop}
\begin{proof}
Let $\ga^{\rm a}\subset\gl_n(\kk)$ denote the algebraic hull of
$\ga$ (the smallest algebraic subalgebra containing $\ga$).
We have $[\ga,\ga]=[\ga^{\rm a},\ga^{\rm a}]$,
cf.~Chevalley \cite[Proposition 1]{Chevalley47}.
Let $\ha'\subset\ga^{\rm a}$ stand for the annihilator of $o$.
Then $\ha'$ is an algebraic subalgebra of~$\ga^{\rm a}$.

Let $\aa=\ga^{\rm a}/[\ga,\ga]$.
Consider the canonical map $\pi:\ga^{\rm a}\to\aa$.
Since an algebraic subalgebra of a commutative algebraic Lie algebra has a
complementary algebraic subalgebra, also defined over $\kk$, there is an
algebraic subalgebra $\ha_1\subset\aa$ such that
$\aa=\pi(\ha')\oplus\ha_1$.
Set $\wt\ga=\pi^{-1}(\ha_1)$ and $\wt\ha=\wt\ga\cap\ha'$.
We have
\[
\wt\ha=[\ga,\ga]\cap\ha'=[\ga,\ga]\cap\ha.
\]
The fact that $\wt\ga$ is \pre\ follows from
\[
\dim\wt\ga=\dim\ga^{\rm a}-\dim\pi(\ha')=n+\dim\ha'-\dim\pi(\ha')=n+\dim\wt\ha.
\qedhere
\]
\end{proof}

\begin{cor}\label{cor:Levi}
For every \et\ Lie algebra there exists an algebraic \et\ Lie
algebra over $\kk$ with the same derived subalgebra (and the
same maximal semisimple subalgebra).
\end{cor}

If $X\mapsto Xo$ is an isomorphism over $\kk$ then it is such over
any extension field of $\kk$. Hence:

\begin{prop}\label{prop:base_change}
Let $\ol{\kk}$ denote an extension field of $\kk$.
A Lie algebra $\ga\subset\gl_n(\kk)$ is \et\ if and only if 
$\ga\otimes_{\kk}\ol{\kk}\subset\gl_n(\ol{\kk})$ is \'etale.
\end{prop}

\section[Non-existence for groups with simple factors $\F_4$ and $\E_8$]{Non-existence of \'etale modules for groups with simple factors $\F_4$ or $\E_8$}
\label{sec:no_F4E8}

For an arbitrary Lie group $G$ to have a (real, finite-dimensional)
\emph{\'etale} module $V$ means that $G$ has an open orbit
in $V$ in the standard topology of $V$ and $\dim G=\dim V$.
We use the results of the previous section and the Sato-Kimura
classification of algebraic prehomogeneous modules to establish the
following non-existence result:

\begin{thm}\label{thm:no_F4E8}
Let $G$ be a real Lie group with Lie algebra
$\ga$ and a linear action on a finite-dimensional real vector
space $V$.
If the module $(G,V)$ is \'etale, then a maximal semisimple
subalgebra of $\CC\otimes\ga$ does not contain simple factors
$\fa_4$ or $\ea_8$.
\end{thm}

The proof needs some preparations.

\begin{prop}\label{prop:reduce}
Let $G$ be a linear algebraic group.
Given a short exact sequence of $G$-modules
\begin{equation}\label{eqn:exact}
0\longrightarrow U\longrightarrow V\stackrel{\pi}{\longrightarrow} W\longrightarrow 0
\end{equation}
where $V$ is prehomogeneous with a point $o$ in general position,
let $G'$ be the stabilizer in $G$ of the line spanned by $\pi(o)\in W$. Then $G'$ preserves $U':=U+\langle o\rangle$ and has an open orbit on it. Moreover, the stabilizer $H$ of $o$
in $(G,V)$ is also the stabilizer of $o$ in $(G',U')$.
\end{prop}
\begin{proof}
Note that $o\nin U$ since $o$ is in general position.
The fact that $G'$ preserves $U'$ follows immediately from definitions and the property $\Ker\pi=U$. Note that $H\subset G'$. It remains to show that the orbit $G'o\subset U'$ is open. Since $(G,V)$ is \pre, so is $(G,W)$. Hence, the action of $G$ on the
projective space $\P W$ over $W$ has an open orbit and its generic stabilizer is conjugate to~$G'$. We conclude $\dim G-\dim G'=\dim W-1$, and therefore
\begin{align*}
\dim G'o=\dim G'-\dim H
&=\dim V- (\dim G-\dim G')\\
&=\dim V-\dim W+1=\dim U'.
\qedhere
\end{align*}
\end{proof}

\begin{lemma}\label{lem:R_orbit}
Let $(G,V)$ be a prehomogeneous module for an algebraic group
$G$ with solvable radical $R$.
Assume there exists an irreducible submodule $U$ of
codimension $1$ in $V$ that is not a direct summand of $V$.
Then $(R,V)$ is prehomogeneous.
\end{lemma}
\begin{proof}
Let $W=V/U$ denote the one-dimensional quotient module for
$G$. Note that $W$ is prehomogeneous since $V$ is.
Let $\Ru(G)$ denote the unipotent radical of $G$, and
$A$ the center of $L(G)$, so that $R=A\cdot \Ru(G)$ and
$L(G)=A\cdot S(G)$.
Let $\ra=\Lie R$, $\rau=\Lie\Ru(G)$ and $\aa=\Lie A$.
Let $x$ be a non-zero point in a one-dimensional $L(G)$-invariant
complementary subspace $W'$ to $U$ in $V$.
We will show that $Rx\subset V$ is open.
It suffices to show that $\ra{x}=V$.

Note that $\Ru(G)$ acts trivially on $U$, as the
eigenspace for eigenvalue $1$ of $\Ru(G)$ is $G$-invariant
and non-zero, hence all of $U$ by irreducibility.
It follows that $U$ is $L(G)$-irreducible.
Also, both $\Ru(G)$ and $S(G)$ act trivially on the
one-dimensional module $W$.

Since $U$ is not a direct summand in $V$,
$\rau x$ is a non-zero subspace of $U$.
Moreover, $\rau x$ is $L(G)$-invariant, and hence coincides
with $U$, since the latter is $L(G)$-irreducible.
%
Since $\Ru(G)$ and $S(G)$ act trivially on $W$,
the prehomogeneity of $W$ requires that $A$ acts non-trivially
on $W$ and hence on $x$.
So $\aa x= W'$, and it follows that
$\ra x=\aa x+\rau x= W'+U=V$.
\end{proof}

Given a linear algebraic group $G$ and a rational $G$-module
$V$, we call $(G,V)$ \emph{casual}\footnote{It is called \emph{trivial} in \cite[Definition 5, p.~43]{SK77}. We decided to use another term to avoid confusions.}
if it is equivalent to $(G'\times\GL_n, V'\otimes\C^n)$ for
an algebraic subgroup $G'\subset\GL(V')$ and $n\geq\dim V'$.
All such modules are \pre\ with generic stabilizer $H$ satisfying
$L(H)\cong L(G')\times\GL_{n-\dim{V'}}$, and the irreducible ones
are given by cases I (1) and III (1)
in the Sato-Kimura classification \cite[\S 7]{SK77}.

\begin{rmk}\label{rem:casual}
A module that is equivalent to
a casual irreducible \'etale module is necessarily
equivalent to $(F\times\GL_n,\C^n\otimes\C^n)$ for some finite
group $F$ acting irreducibly on $\C^n$.
If $(G,V)$ is castling-equivalent to such a module, then it
follows immediately that all simple factors of $S(G)$
are special linear groups.
\end{rmk}



\begin{prop}\label{prop:noncasual}
Let $(G,V)$ be an \'etale module for a linear
algebraic group $G$, and let $Q$ be a simple factor of
$S(G)$ not isomorphic to $\SL_n$ for any $n$.
There exists an \'etale module $(\wt{G},\wt{V})$ with a simple factor
$\wt{Q}\cong Q$ in $S(\wt{G})$ and an irreducible quotient module
$\wt{W}$ of $\wt{V}$ such that $\wt{Q}$ acts non-trivially on $\wt{W}$
and $\wt{W}$ is not castling-equivalent to a casual module.
\end{prop}
\begin{proof}
We prove the claim by induction on $\dim V$.
Note that since $\dim Q > 1$ the module cannot be \'etale in the case
$\dim V=1$, so the claim holds trivially.

Suppose now that $\dim V\geq 2$.
If $V$ is irreducible, then $(G,V)=(\wt{G},\wt{V})$ satisfies the
claim in light of Remark \ref{rem:casual}.
So we may further assume that $V$ is not irreducible.

Assume that there is an irreducible quotient $W=V/U$ with $\dim W\geq 2$. If $Q$ acts non-trivially on $W$ and $(G,W)$ is not castling-equivalent to a casual module, we can put $(\wt{G},\wt{V}):=(G,V)$ and $\wt{Q}:=Q$.
Otherwise, either $Q$ acts trivially on $W$ or $(G,W)$ is castling-equivalent to a casual module. Then $S(G_x)$ contains a factor isomorphic to $Q$, where $x\in W$ is a point in general
position.
In this case, if $(G',U')$ is as in Proposition
\ref{prop:reduce}, then $(G',U')$ is \'etale and $G'$ contains a
conjugate of $Q$.
Since $\dim U'=1+\dim U<\dim V$, the claim now follows by
induction on $\dim V$.

Suppose now that all irreducible quotients of $V$ are
one-dimensional, and let $W=V/U$ be one of them.
There exists a maximal proper
submodule $U_0\subset U$, so that $W_0:=U/U_0$ is 
irreducible, and for $W_1:=V/U_0$ we have the exact sequence
\[
0\longrightarrow W_0\longrightarrow W_1\longrightarrow W\longrightarrow 0.
\]
Note that $W_1$ is prehomogeneous since $V$ is.
We claim that the solvable radical $R$ of $G$ has an open orbit in
$W_1$.
If $W_0$ is a direct summand in $W_1$, then by the assumption that all quotients of $V$ are one-dimensional, $\dim W_0 = 1$, and therefore $S(G)$ acts trivially on $W_1$,
implying that the open $G$-orbit is also an open $R$-orbit.
Suppose $W_0$ is not a direct summand in $W_1$.
Since $W$ and $W_0$ are both irreducible, we can apply
Lemma \ref{lem:R_orbit} (with $V$ replaced by $W_1$) to
conclude that $R$ has an open orbit in $W_1$.
Therefore, $S(G)$ belongs to the stabilizer of a point in general
position in $W_1$.
We can now use Proposition \ref{prop:reduce} (with $W$, $U$ replaced by
$W_1$, $U_0$) and induction to derive the statement.
\end{proof}

\begin{proof}[Proof of Theorem \ref{thm:no_F4E8}]
If $(G,V)$ is a real \'etale module, then by Proposition
\ref{prop:base_change} there exists a complex
\'etale module $(G_\CC,V_\CC)$ where $G_\CC$ is a Lie group
with Lie algebra $\CC\otimes\ga$.
So by Proposition \ref{prop:hull}, we may assume that
$(G,V)$ is a complex algebraic \'etale module.
According to the classification of irreducible prehomogeneous
modules for reductive algebraic groups \cite[\S 7]{SK77},
all irreducible pre\-homogeneous modules for reductive algebraic
groups with $\F_4$ or $\E_8$ as a simple factor are
castling-equivalent to a casual module.
It remains to apply Proposition \ref{prop:noncasual} to $(G,V)$.
\end{proof}

\bibliographystyle{myspmpsci}
\bibliography{paper}
\end{document}